\documentclass{elsarticle}
\usepackage{amsmath}
\usepackage{graphicx}
\usepackage{hyperref}
\usepackage{mathtools}
\usepackage{amsfonts}
\usepackage{amssymb}
\usepackage{color}
\usepackage{amsthm}
\usepackage{fullpage}
\usepackage{enumitem}

\newtheorem{theorem}{Theorem}[section]

\newtheorem{lemma}[theorem]{Lemma}

\theoremstyle{definition}
\newtheorem{definition}[theorem]{Definition}

\theoremstyle{remark}
\newtheorem{remark}[theorem]{Remark}

\numberwithin{equation}{section}

\newcommand{\F}{\mathbb{F}_q}
\newcommand{\Fm}{\mathbb{F}_{q^m}}

\newcommand{\Ll}{\mathfrak{L}}
\newcommand{\N}{\mathfrak{N}}
\DeclareMathOperator{\Tr}{Tr}
\newcommand{\Oo}{\mathcal O}
\DeclareMathOperator{\Div}{Div}
\DeclareMathOperator{\Prin}{Prin}
\DeclareMathOperator{\Cl}{Cl}
\DeclareMathOperator{\Int}{Int}
\DeclareMathOperator{\supp}{supp}
\DeclareMathOperator{\Ord}{Ord}

\newcommand{\p}{\mathfrak{p}}
\newcommand{\aaa}{\mathfrak{a}}
\newcommand{\bb}{\mathfrak{b}}

\begin{document}

\title{An estimate for incomplete mixed character sums and applications}

\author[1]{Arpan Chandra Mazumder\fnref{fn1}}
\ead{arpan10@tezu.ernet.in}

\author[2]{Giorgos Kapetanakis\corref{cor1}}
\ead{kapetanakis@uth.gr}

\author[3]{Sushant Kala}
\ead{sushant@imsc.res.in}

\author[1]{Dhiren Kumar Basnet}
\ead{dbasnet@tezu.ernet.in}

\fntext[fn1]{The first author is supported by DST INSPIRE Fellowship, under grant no. DST/INSPIRE Fellowship/2021/IF210206.}
\cortext[cor1]{Corresponding author}

\affiliation[1]{organization={Department of Mathematical Sciences, Tezpur University}, city={Tezpur}, state={Assam}, postcode={784028}, country={India}}

\affiliation[2]{organization={Department of Mathematics, University of Thessaly}, addressline={3rd km Old National Road Lamia-Athens}, postcode={35100}, city={Lamia}, country={Greece}}

\affiliation[3]{organization={Department of Mathematics, Institute of Mathematical Sciences (HBNI)}, addressline={CIT Campus, IV Cross Road}, city={Chennai}, postcode={600113}, country={India}}

	\begin{abstract}
		Let $q$ be a prime power and $m>1$ be any integer. Let $\Fm$ be the finite field of order $q^m$ and $\theta\in\Fm$ be such that $\Fm = \F(\theta)$. We obtain a nontrivial bound for the mixed character sum $\sum_{x \in\F}\chi(\theta+x)\psi(x)$, where $\chi$ and $\psi$ are multiplicative and additive characters of $\Fm$ and $\F$, respectively, using function field methods. As an application of our main result, we prove that for fixed $m$ and sufficiently large prime powers $q$, that satisfy certain conditions, $\Fm/\F$ possesses the weak line property for primitive normal elements. 
		In particular, our result is a strengthening of existing results.
	\end{abstract}

	
	\maketitle
	
	\section{Introduction}
	
	Throughout the article, let $p$ be a prime, $q$ be a power of $p$ and $m$ be a positive integer. Denote by $\F$, the finite field of order $q$ and characteristic $p$ and by $\Fm$, the extension field of $\F$ of degree $m$. The multiplicative group $\Fm^*$ is cyclic and a generator of this group is called a \emph{primitive} element of $\Fm$. An element $\theta \in \Fm$ is said to be \emph{normal} over $\F$ if the set of all its conjugates with respect to $\F$, that is, the set $\{\theta, \theta^{q}, \ldots , \theta^{q^{m-1}}\}$ forms a basis of $\Fm$ over $\F$. An element $\theta \in \Fm$ is said to be \emph{primitive normal} if it is both primitive and normal over $\F$.  
	
	Let $\Fm$ be an extension of degree $m$ over $\F$. An element $\theta\in\Fm$ is called a \emph{generator} of the extension $\Fm/\F$, if $\Fm = \F(\theta)$. For any generator $\theta$ of $\Fm/\F$, we call the set $\{\theta + x : x \in \F\}$, the \emph{set of translates} of $\theta$ over $\F$ (since geometrically the set can be interpreted as a translate of $\F$). For any particular extension $\Fm/\F$, we refer to the question of the existence of elements of $\Fm$ of a certain type and of the form $\theta + x$, where $x \in \F$, for every generator $\theta$ of $\Fm$ as the \emph{translate problem} and, if the answer is positive, we say that the extension $\Fm/\F$ possesses the \emph{translate property} ($TP$) for this type of elements. 
	
	Likewise, for any generator $\theta$ of $\Fm/\F$ and $\alpha \in \Fm^*$, we call the set $\{\alpha(\theta + x) : x \in \F\}$ the \emph{line} of $\theta$ and $\alpha$ over $\F$ (since this set can be interpreted as a line in $\Fm$). For any particular extension $\Fm/\F$, we refer to the question of the existence of elements of $\Fm$ of a certain type and of the form $\alpha(\theta + x)$, where $x \in \F$, for every $\alpha \in \Fm^*$ and for every generator $\theta$ of $\Fm$ as the \emph{line problem} and, if the answer is positive, we say that the extension $\Fm/\F$ possesses the \emph{line property} ($LP$) for this this type of elements. 

Finally, if in the above setting we only consider the case $\alpha\in\F^*$, we refer to the \emph{weak line problem} and we say that the extension $\Fm/\F$ possesses the \emph{weak line property} ($WLP$) for this type of elements, respectively. So, once we observe that $\alpha=1$ corresponds to the translate property and that the conditions $\alpha\in\Fm^m$, $\alpha\in\F^*$ and $\alpha=1$ are, in terms of strength, in ascending order, it is clear that
\[ LP \Rightarrow WLP \Rightarrow TP . \]

It is noteworthy that R\'ua~\cite{rua15,rua17} successfully linked the translate property for primitive elements of a finite field extension with the existence of primitive elements of \emph{semifields}, that is, a finite nonassociative division ring with multiplicative identity element, sparking additional interest in the translate property (and its generalizations) for finite field extensions.
	
	In 1937, Davenport~\cite{davenport1937} provided the bound \[ \left| \sum_{x \in \mathbb{F}_p}\chi(\theta +x)\right| \ll p^{1-\frac{1}{2(m+1)}} \] where $\chi$ is a non-trivial multiplicative character of $\F$, $m>1$ any integer and $\F=\mathbb{F}_{p^m}=\mathbb{F}_p(\theta)$. As a consequence established that extensions over (sufficiently large) prime fields possess the translate property for primitive elements.
	In 1953, Carlitz~\cite{carlitz1953} extended the above to arbitrary finite field extensions (not necessarily over prime fields). 
	The following is a classical result in the study of primitive elements.
	
	\begin{theorem}[Davenport-Carlitz]\label{devcar}
		Let $m>1$  be an integer. There exists some $TP(m)$ such that, for every prime power $q > TP(m)$, the extension $\Fm/\F$ possesses the translate property for primitive elements.
	\end{theorem}
	
	In 1989, Katz~\cite{katz} provided the following estimate for a class of character sums that occur as eigenvalues of adjacency matrices of certain graphs constructed by Chung~\cite{chung}. 
	
	\begin{theorem}[Katz] \label{katzbd}
		Let $\chi$ be any nontrivial multiplicative character of $\Fm^*$ and $\theta$ in $\Fm$, a generator $\Fm/\F$. Then 
\[
			\left|\sum_{x \in \F}\chi(\theta +x)\right|\leq (m-1)q^{1/2} .
		\]
	\end{theorem}
	
	
	Furthermore, by using this result, a natural generalization of Theorem~\ref{devcar} was established by Cohen~\cite{cohen2010} in 2010.
	
	\begin{theorem}[Cohen]\label{cohen1}
		Let $m>1$  be an integer. There exists some $LP(m)$ such that, for every prime power $q > LP(m)$, the extension $\Fm/\F$ possesses the line property for primitive elements.
	\end{theorem}
	
Recently, Cohen and Kapetanakis~\cite{cohen2021} extended Theorems~\ref{devcar} and \ref{cohen1} to \emph{$r$-primitive} elements, that is, elements of $\Fm$ of multiplicative order $(q^m-1)/r$, where necessarily $r\mid q^m-1$.
	
	\begin{theorem}[Cohen-Kapetanakis]
		Let  $m>1$ and $r$ be integers. There exists some $LP_r(m)$ such that, for every prime power $q > LP_r(m)$ with the property $r \mid q^m-1$, the extension $\Fm/\F$ possesses the line property for $r$-primitive elements. As for the translate property for $r$-primitive elements, the same is true for some $TP_r(m) \leq LP_r(m)$.
	\end{theorem}
	
	Motivated by this series of outstanding results on the aforementioned line and translate properties for primitive elements over finite fields  	 and the estimates for character sums that served as the backbone behind these results, we consider the line and translate properties for primitive normal elements. 
	Towards this end, observe that the results of Davenport~\cite{davenport1937} and Katz~\cite{katz} involve only a multiplicative character. However, since normality is a property of the additive structure of $\Fm$, we will encounter character sum estimates that are mixed, in the sense that an additive character will also be present.
	
	In 1990, Perel{'}muter and Shparlinski~\cite{perel2} showed that, if $\theta$ is a generator of the extension $\mathbb F_{p^m}/\mathbb F_p$ and $a\in\mathbb F_{p}$, then
	\begin{equation}\label{perelbd}
		\left|\sum_{t \in \mathbb F_p} \chi(\theta+t) \exp (2 \pi i a t / p)\right| \leq m p^{1 / 2},
	\end{equation}
	where $\chi$ is a nontrivial multiplicative character of $\mathbb{F}_{p^m}$. Note that the aforementioned result applies to finite field extensions over prime fields and, although $\psi(t) = \exp (2 \pi i a t / p)$ is an additive character, it is not an arbitrary one.
	
	In 2014, Fu and Wan~\cite{fuwan14} obtained the following.
	\begin{theorem}[Fu-Wan] \label{thm:fuwan}
	  Take $f,g\in\Fm(x)$ written in their simplest form. Further, set $D_1 = \deg(f)$, $D_2 = \max (\deg(g),0)$, $D_3$ the degree of the denominator of $g$, and $D_4$ the degree of the part of the denominator of $g$ that is relatively prime to both the enumerator and the denominator of $f$. Let $\chi$ be a multiplicative character of $\Fm$ and $\psi$ a nontrivial additive character of $\F$. If $\Tr_{q^m/q}(g)$ is not of the form $r^q-r$, for some $r\in\overline{\F}(x)$, then
	  \[ \left| \sum_{t\in\F, f(t)\neq 0,\infty, g(t)\neq\infty} \chi(f(t)) \psi(\Tr_{q^m/q}(g(t))) \right| \leq (m(D_1+D_3+D_4)+D_2-1)q^{1/2} . \]
	\end{theorem}
	As a consequence, one obtains, that, given $\theta$ a generator of $\Fm/\F$, $\chi$ a multiplicative character of $\Fm$ and $\psi$ a nontrivial additive character of $\F$, then
	\begin{equation} \label{eq:maineq}
	\left| \sum_{t\in\F} \chi(\theta+t)\psi(t) \right| \leq m q^{1/2} .
	\end{equation}
	However, it is worth pointing out that Fu and Wan~\cite{fuwan14}, in the proof of Theorem~\ref{thm:fuwan}, employ $\ell$-adic cohomology theory, whilst Perel{'}muter and Shparlinski~\cite{perel2}, in the proof of Eq.~\eqref{perelbd}, used simpler function field arguments.
	With all of the above in mind, the goal of this paper is first prove Eq.~\eqref{eq:maineq}, using simpler function field arguments (and also covering the case of $\chi$ being nontrivial and $\psi$ trivial), as follows.
		\begin{theorem}\label{mainthm}
		Let $\chi$ be a multiplicative character of $\Fm$ and $\psi$ an additive character of $\F$, respectively, such that not both of them are trivial. Further, let $\theta$ be a generator of the extension $\Fm/\F$. Then Eq.~\eqref{eq:maineq} holds.
	\end{theorem}
	At this point we stress that the case $\psi$ being trivial and $\chi$ nontrivial is covered by Theorem~\ref{katzbd}. Our approach essentially generalizes this result, whilst  being simpler and more elementary.
	Furthermore, it is noteworthy that Theorem~\ref{mainthm} was recently used in \cite{garefalakiskapetanakis26} in the latest attack on the Morgan-Mullen conjecture \cite{morganmullen}.
	 
 Then, we use Theorem~\ref{mainthm} in order to explore whether the extension $\Fm/\F$ possesses the line property for primitive normal elements. Namely, we obtain the following.
 
 \begin{theorem} \label{lineprop}
Fix a prime $m$. There exists an integer $WLPN(m)$,
such that for every prime power $q\geq WLPN(m)$ that is primitive modulo $m$, the extension $\Fm/\F$ possesses the weak line property for primitive normal elements. The same is true for some $TPN(m)\leq WLPN(m)$ regarding the translate property for primitive normal elements.
\end{theorem}
 	
	Our paper is organized as follows. Section~\ref{preliminaries} provides background material
	that is used along the way. In Section~\ref{main_proof} we extend the mixed character sum estimate of \eqref{perelbd} to all prime powers $q$, whilst also considering the case when $\chi$ is trivial, obtaining Theorem~\ref{mainthm}. In Section~\ref{line}, we apply this bound to prove Theorem~\ref{lineprop_gen}, a technical generalization of Theorem~\ref{lineprop}. In Section~\ref{sec:proof_lineprop}, we use Theorem~\ref{lineprop_gen} to obtain Theorem~\ref{lineprop}. We conclude with possible extensions of our work in Section~\ref{further_research}.
	
	\section{Preliminaries} \label{preliminaries}
	This section presents background information essential for establishing our results. We refer the interested reader to established textbooks such as \cite{lidlniederreiter97,mullen2013handbook} for further details and detailed proofs of the facts presented in the first three subsections of this section.
	\subsection{The trace function}
	We begin with the \emph{trace} of the extension $\Fm/\F$, that is, the function
	\[ \Tr_{q^m/q} : \Fm \to \F; x\mapsto \sum_{i=0}^{m-1} x^{q^i} . \]
	It is well-known that the trace function is surjective, $\F$-linear and transitive, that is, for $a,b\in\F$ and $x,y\in\Fm$, we have that
	\[ \Tr_{q^m/q} (ax+by) = a\Tr_{q^m/q}(x) + b\Tr_{q^m/q}(y) \]
	and, if we have the finite field tower $\F \subseteq \mathbb F_{q^n} \subseteq \Fm$, then, for $x\in\Fm$,
	\[ \Tr_{q^m/q}(x) = \Tr_{q^n/q}\left( \Tr_{q^m/q^n}(x) \right) . \]
	\begin{remark}
	Within the context of a line, the trace function behaves in a very predictable way. To see this, take the function
	\[ \Tr_{q^m/q} : \{ \alpha(\theta + x) , x\in\F\} \to \F ; a\mapsto \Tr_{q^m/q}(a) \]
	and observe that $\Tr_{q^m/m}(\alpha(\theta + x)) = \Tr_{q^m/m}(\alpha\theta) + x \Tr_{q^m/m}(\alpha)$. Thus, if $\Tr_{q^m/q}(\alpha)\neq 0$,
	it is a bijection,
	while it is constant otherwise. In other words, the line property for elements of prescribed trace is trivial and, as such, it cannot be combined with other properties, such as normality, in a meaningful way.
	\end{remark}
	
	\subsection{Finite field characters}
	
	In our work, the characteristic functions of primitive and normal elements play an important role. To represent those functions, the idea of character of finite Abelian group is necessary.
	
	\begin{definition}
		Let $\mathbb{G}$ be a finite Abelian group. A character $\chi$ of the group $\mathbb{G}$ is a homomorphism from $\mathbb{G}$ to $\mathbb{S}^1$, where $\mathbb{S}^1:= \lbrace z\in \mathbb{C}: |z| = 1 \rbrace$ is the multiplicative group of complex numbers of unit modulus, i.e. $\chi(a_1a_2)= \chi(a_1)\chi(a_2)$, for all $a_1, a_2 \in \mathbb{G}$. The special character $\chi_0: \mathbb{G} \rightarrow \mathbb{S}^1$ defined as  $\chi_0(a) = 1$, for all $a \in \mathbb{G}$ is called the trivial character of $\mathbb{G}$.
	\end{definition}
	
	In a finite field $\mathbb{F}_{q^m}$ there are two main group structures, one is the additive group $\mathbb{F}_{q^m}$ and the other is the multiplicative group $\mathbb{F}^*_{q^m}$. Therefore we have two types of characters pertaining to these two group structures, one is the arbitrary \emph{additive character} for $\mathbb{F}_{q^m}$ denoted by $\psi$ and the other one is the arbitrary \emph{multiplicative character} for $\mathbb{F}^*_{q^m}$ denoted by $\chi$. The \emph{canonical} character of $\Fm$ is defined as
	\[ \psi_1 : \Fm \to \mathbb{S}^1 ; ~a \mapsto e^{2\pi i \Tr_{q^m/p}(a)/p} , \]
	and it is an additive character of $\Fm$, while all the additive characters of $\Fm$ are of the form
	\[ \psi_c : \Fm \to \mathbb{S}^1 ; ~a \mapsto \psi_1(ca) , \]
	for some $c\in\Fm$, see Theorem~5.7 of \cite{lidlniederreiter97}.
	The multiplicative characters associated with $\mathbb{F}^*_{q^m}$ are extended from $\mathbb{F}^*_{q^m}$ to $\mathbb{F}_{q^m}$ by the rule
	\[ \chi(0):=\begin{cases}
		0,& \text{if } \chi\neq\chi_0,\\
		1,& \text{if } \chi=\chi_0. 
	\end{cases} \]
	
	\begin{definition}
		The \emph{$\mathbb{F}_q$-order} of an additive character $\psi \in \widehat{\mathbb{F}_{q^m}}$ is the monic $\mathbb{F}_{q}$-divisor $g$ of $x^m-1$ of minimal degree such that $\psi \circ g$ is the trivial character of $\widehat{\mathbb{F}_{q^m}}$, where ($\psi \circ g)(\alpha):= \psi(g \circ \alpha)$, and $g\circ \alpha = \sum_{i=0}^{n} a_i \alpha^{q^i}$ if $g(x)=\sum_{i=0}^{n} a_i x^{i}$ for any $\alpha \in \Fm$ and is denoted by $\Ord_q(\psi)$.
	\end{definition}
	
	It is well-known that, if $z\mid q^m-1$ there are exactly $\phi(z)$ multiplicative characters of order $z$, where $\phi$ stands for the Euler's phi function and the \emph{order} of a multiplicative character is defined naturally. Likewise, if $g\mid x^m-1$, there are exactly $\Phi(g)$ additive characters of $\F$-order $g$, where $\Phi$ stands for the Euler's function in the ring $\F[x]$.
	
		In the proceeding sections, we will need to focus a bit further into the nature of the additive characters of $\Fm$. The following lemma will be useful.
	
	\begin{lemma}\label{lemma:constant_char}
	Let $\psi$ be an additive character of $\Fm$. The restriction of $\psi$ in $\F$, $\psi|_{\F}$, is an additive character of $\F$. Also, the following are equivalent:
	\begin{enumerate}
	\item $\psi|_{\F}$ is the trivial additive character of $\F$. \label{it:ip1}
	\item $\psi = \psi_c$ for some $c\in\Fm$ such that $\Tr_{q^m/q}(c) = 0$. \label{it:ip2}
	\item $\Ord_q(\psi) \mid \frac{x^{m}-1}{x-1}$. \label{it:ip3}
	\end{enumerate}
	\end{lemma}
	\begin{proof}
	It is obvious that $\psi|_{\F}$ is an additive character of $\F$. Further, suppose that $\psi = \psi_c$.
	
	Now we will show that \eqref{it:ip1} $\iff$ \eqref{it:ip2}. Take some $a\in\F$. Then
	\[ \psi(a) = \psi_c(a) = \psi_1(ca) = e^{2\pi i \Tr_{q^m/p}(ca)/p } = e^{2\pi i \Tr_{q/p}(\Tr_{q^m/q}(ca)) /p} = e^{2\pi i \Tr_{q/p}\left( \Tr_{q^m/q}(c) \cdot a \right)/p} . \]
	In other words, when viewed as a character of $\F$, $\psi|_{\F} = \psi_{\Tr_{q^m/q}(c)}$ and the result follows.
	
	Finally, we focus on the equivalence \eqref{it:ip2} $\iff$ \eqref{it:ip3}. We have that
	\begin{align*}
	\Ord_q(\psi_c) \mid (x^{m}-1)/(x-1) & \iff \text{for every } a\in\Fm,\ \psi_c(\Tr_{q^m/q}(a)) = 1 \\
	 & \iff \text{for every } a\in\Fm,\ \psi_1(c \Tr_{q^m/q}(a)) = 1 \\
	 & \iff \text{for every } b\in\F,\ \Tr_{q^m/p}\left( c b\right) = 0 \\
	 & \iff \text{for every } b\in\F,\ \Tr_{q/p}\left(\Tr_{q^m/q}\left( c b\right)\right)  = 0\\
	 & \iff \text{for every } b\in\F,\ \Tr_{q/p}\left(b\Tr_{q^m/q}( c )\right)  = 0 \\
	 & \iff \Tr_{q^m/q}(c) = 0 . \qedhere
	\end{align*}
	\end{proof}
	
	\subsection{Freeness}
	
	The notion of freeness on both the multiplicative and the algebraic structure of finite field elements plays a vital role in this work.
	
		\begin{definition}
		Let $e\mid q^m-1$, then an element $\alpha \in \mathbb{F}^{*}_{q^m}$ is called \emph{$e$-free}, if $d\mid e$ and $\alpha = y^d$, for some $y \in \mathbb{F}^{*}_{q^m}$ imply $d=1$. It is clear from the definition that an element $\alpha \in \mathbb{F}^{*}_{q^m}$ is primitive if and only if it is a $(q^m-1)$-free element.  
	\end{definition}
	
		The additive group of $\mathbb{F}_{q^m}$ is an $\mathbb{F}_{q}[x]$-module under the rule \[f\circ\alpha := \sum_{i=0}^n a_i\alpha^{q^i},\] for $\alpha\in \mathbb{F}_{q^m}$, where $f(x)= \sum_{i=0}^na_ix^i \in \mathbb{F}_{q}[x]$. For $\alpha \in \mathbb{F}_{q^m}$, the \emph{$\mathbb{F}_q$-order} of $\alpha$ is the monic $\mathbb{F}_q$-divisor $g$ of $x^m-1$ of minimal degree such that $g \circ\alpha=0$ and is denoted by $\Ord_q(\alpha)$.
	
	\begin{definition}
		For $g\mid x^m-1$, an element $\alpha\in \mathbb{F}_{q^m}$ is called \emph{$g$-free} if $\alpha = h \circ \beta$ for some $\beta \in \mathbb{F}_{q^m}$ and $h\mid g$ imply $h=1$. From the definition, it is obvious that an element $\alpha \in \mathbb{F}_{q^m}$ is normal if and only if it is $(x^m-1)$-free. It is clear that $x^m-1$ can be freely replaced by its $p$-radical $g_0:= x^{m_0}-1$, where $m_0$ is such that $m=m_0p^a$, here $a$ is a nonnegative integer and $\gcd (m_0, p)=1$.
	\end{definition}
	
	For any $e\mid q^m-1$, the characteristic function of $e$-free elements of $\mathbb{F}^{*}_{q^m}$ can be expressed as follows.
	\begin{equation}\label{e-free ch}
		\rho_e: \mathbb{F}^{*}_{q^m}\rightarrow \{0,1\}; \alpha \mapsto \epsilon(e) \sum_{d\mid e} \left( \frac{\mu(d)}{\phi(d)} \sum_{(d)} \chi_d(\alpha) \right),
	\end{equation}
	where $\epsilon(e):= \phi(e)/e$, $\mu$ is the M\"obius function and $\chi_d$ stands for any character in $\widehat{\mathbb{F}_{q^m}^*}$ of order $d$ and the inner sum runs through all characters of order $d$.
	
	Similarly, for any $g\mid x^m-1$, the characteristic function of $g$-free elements of $\mathbb{F}_{q^m}$ can be expressed as follows:
	\begin{equation}\label{g-free ch}
		\kappa_g : \mathbb{F}_{q^m}\rightarrow \{0, 1\}; \alpha \mapsto \varepsilon(g) \sum_{f\mid g} \left( \frac{\mu^\prime(f)}{\Phi(f)} \sum_{(f)} \psi_f(\alpha) \right),
	\end{equation}
	where $\varepsilon(g):= \Phi(g)/{q^{\deg(g)}}$, $\psi_f$ stands for any character in $\widehat{\mathbb{F}_{q^m}}$ of $\mathbb{F}_{q}$-order $f$ and  $\mu^\prime$ is the analogue of the M\"obius function defined as follows:
	\[
	\mu^\prime(g):=\begin{cases}
		(-1)^s , & \text{if $g$ is the product of $s$ distinct irreducible monic polynomials}, \\
		0 , &\text{otherwise},
	\end{cases}
	\]
	and the inner sum runs through the additive characters of $\F$-order $f$.
	
	A set of elements of $\Fm$ that will be of interest for us is $(x^m-1)/(x-1)$-free elements. It is not hard to check that these elements do not belong to any intermediate extension of $\Fm/\F$, while, if $p\mid m$, then they are exactly the elements that are normal over $\F$. Additionally, if $m$ is a prime such that $q$ is \emph{primitive modulo $m$}, i.e., it generates the group $\mathbb Z_m^*$, then these elements are easily characterized as follows.
	
	\begin{lemma} \label{lemma:g-free}
	Suppose $m$ is a prime and $q$ is primitive modulo ${m}$. Then the $(x^m-1)/(x-1)$-free elements of $\Fm$ are exactly those that do not belong to any intermediate extension of $\Fm/\F$.
	\end{lemma}
	\begin{proof}
	Set $g(x) := (x^m-1)/(x-1)$. The restrictions on $m$, along with Theorems~2.45 and 2.47 of \cite{lidlniederreiter97}, yield that $g$ is irreducible, hence
	\[ x^m-1 = g(x) (x-1) \]
	is the prime decomposition of $x^m-1$. Now, take some $a\in\Fm$ that is $g$-free. Theorem~5.4 of \cite{huczynskamullenpanariothomson13} yields that $g\mid \Ord_q(a)$, that is, $\Ord_q(a) = g$ or $\Ord_q(a) = x^m-1$. The desired result follows.
	\end{proof}

    \subsection{Divisors and Class Groups} \label{subsec:div}
    
    Let $F/K$ be a field extension, such that, for some $x\in F$, $x$ is transcendental over $K$ and $F/K(x)$ is finite. In this case, $F/K$ (or simply $F$) is called a \emph{function field of one variable} over $K$. In particular, the function field $K(x)/K$ is called \emph{rational} and if, $K$ is finite, i.e., $K=\F$ for some prime power $q$, the function field $F/K$ is called \emph{global}. 
    
    A \emph{discrete valuation ring} of the function field $F/K$ is a ring $K\subsetneqq \Oo\subsetneqq F$ with the property $a\in\Oo$ or $a^{-1}\in\Oo$ for all $a\in F$. It can be shown that a discrete valuation ring is a local ring, that is, it has a unique maximal ideal, whilst this ideal is principal. Additionally, the maximal ideal $\p$ of the discrete valuation ring $\Oo$ of the function field $F/K$ is called a \emph{place} of $F/K$ and, if $\p = t\Oo$, then each $z\in F^*$ has a unique representation of the form $z=t^{\nu_\p(z)} u$, where $\nu_\p(z) \in\mathbb Z$ and $u\in\Oo^*$. Furthermore, it can be seen that, whilst $t$ above is not unique, the number $\nu_\p (z)$ does not depend on its choice. Also,
    \begin{align*}
    \Oo & = \{ z\in F : \nu_\p(z) \geq 0 \} \cup \{ 0 \}, \\
    \Oo^* & = \{ z\in F : \nu_\p(z) = 0 \} \text{ and} \\
    \p & = \{ z\in F : \nu_\p(z) > 0 \}.
    \end{align*}
  The above facts show a bijection between the places and the discrete valuation rings of $F/K$, so the terms ``place of the valuation ring $\Oo$'' or ``valuation ring of the place $\p$'' make sense. Finally, observe that, since $\p$ is a maximal ideal of $\Oo$, the ring $\Oo/\p$ is a field, thus the number $[\Oo/\p : K]$ is well-defined and it is called the \emph{degree} of the place $\p$ and denoted by $\deg(\p)$. It is known that the degree of a place $\p$ is finite.
    
    An important family of functions of $F$ is that of discrete valuations. Namely, a \emph{discrete valuation} of $F$ is a function
    \[ \nu : F\to \mathbb Z\cup \{\infty \} ; a\mapsto \nu(a) , \]
    such that, for all $a,b\in F$,
    \begin{enumerate}
      \item $\nu(a)=\infty \iff a=0$,
      \item $\nu(ab) = \nu(a) + \nu(b)$,
      \item $\nu(a+b) \geq \min\{ \nu(a),\nu(b) \}$,
      \item there exists some $c\in F$, such that $\nu(c)=1$ and
      \item $\nu(d)=0$ for all $d\in K^*$.
    \end{enumerate}
    In particular, if $\p$ is a place of the function field $F/K$ the function $\nu_\p$ defined above, extended to zero by the rule $\nu_\p(0) = \infty$, is a discrete valuation called the \emph{$\p$-order}.
    
    The free Abelian group generated by the places of $F/K$ called the \emph{divisor group} and denoted by $\Div(F)$ and its elements are called \emph{divisors} and are usually depicted as formal sums, although we will opt for a multiplicative notation. The \emph{unit element of the divisor group} is the divisor $1_{\Div(F)} := \prod_\p \p^0$, where the product runs through the places of $F/K$. Also, the \emph{degree} of the divisor $D = \prod_\p \p^{d_\p}$, where $d_\p\in\mathbb Z$ and almost all of the $d_\p$'s are zero, is defined naturally as the number $\deg(D) := \sum_\p d_\p \deg(\p)$. The \emph{support} of the divisor $D = \prod_\p \p^{d_\p}$ is the set
    \[ \supp(D) := \{ \p : d_\p \neq 0 \} \]
    and it is always a finite subset of the set of places of $F/K$.
    A special family of divisors is that of \emph{principal divisors}, that is, divisors of the form
    \[ [a] = \prod_\p \p^{\nu_\p(a)} , \]
    where $a\in F^*$ and the product runs through the places of $F/K$. It is clear that the principal divisors of $\Div(F)$ form a subgroup of $\Div(F)$, called the \emph{group of principal divisors} and denoted by $\Prin(F)$, while the quotient group
    \[ \Cl(F) := \Div(F)/ \Prin(F) \]
    is called the \emph{divisor class group} of $F$, while two divisors that are in the same class within that group are called \emph{linearly equivalent}. A natural partial ordering is defined in $\Div(F)$ and a divisor $D \geq 1_{\Div{F}}$ (that is, $D = \prod_\p p^{d_\p}$, with $d_\p\geq 0$ for all $\p$) is called \emph{integral}, while the semigroup of integral divisors is denoted by $\Int(F)$.
    
    Next, we turn our attention to the rational function field $K(x)/K$. From the above discussion every place of this function field gives rise to a discrete valuation. However, it is known that in this case there are only two types of discrete valuations:
    \begin{enumerate}
    \item For every irreducible $\pi\in K[x]$, the map
    \[ \nu_\pi : K(x) \to \mathbb Z\cup \{ \infty\} ; f\mapsto \nu_\pi (f) , \]
    where $\nu_\pi(0) = \infty$ and, for $f\neq 0$, $\nu_\pi(f) = n$, where $n$ is the unique integer such that $f = \pi^n g/h$, where $g,h\in K[x]$ are relatively prime to $\pi$, is a discrete valuation of $K(x)$.
    \item The map
    \[ \nu_\infty : K(x) \to \mathbb Z\cup\{ \infty\} ; f \mapsto \nu_\infty (f) , \]
    where $\nu_\infty(0) = \infty$ and, for $f\neq 0$, $\nu_\infty(f) = -\deg(f)$ is a discrete valuation of $K(x)$.
    \end{enumerate}
    It follows that the places of $K(x)$ are exactly the places $\p_\pi := \langle \pi\rangle$, where $\pi\in K[x]$ is a monic irreducible polynomial, called the \emph{finite} places of $K(x)/K$, and $\p_\infty$, called the \emph{infinite} place of $K(x)/K$. Also, one can see that $\deg(\p_\pi) = \deg(\pi)$ and that $\deg(\p_\infty) = 1$. Finally, from the above, it follows that in the rational function field $K(x)/K$, some $f\in K[x]$, if $f=\prod_\pi \pi^{f_\pi}$ is its irreducible decomposition, corresponds (in addition to its corresponding principal divisor) to the divisor
    \[ (f) = \prod_\pi \p_\pi^{f_\pi} \in \Int(K(x)) . \]
    
    For more details on the above facts and detailed proofs we refer the interested reader to Chapter~1 of \cite{stichtenoth09}.

\subsection{Characters modulo a divisor}
We will now define the character modulo an integral divisor. Take some $I\in\Int(F)$. The set $D_I := \{ D\in \Div(F) : \supp(D)\cap\supp(I) = \emptyset\}$ is a subgroup of $\Div(F)$ and the set
\[ E_I := \{ [a] : a \in F^* \text{ and }a\equiv 1\pmod{I} \} \]
is a subgroup of $D_I$ called the \emph{ray} of $I$, where the congruence means that $[a-1]/I \in\Int(F)$.
A homomorphic map $X$ from $\Int(F)$ to a finite and closed under multiplication subset of $\mathbb C$, extended multiplicatively to the whole group $\Div(F)$, is called a \emph{character modulo $I$} (or just a \emph{character}) if
\begin{enumerate}
  \item $G_X := \{D \in\Div(F) : X(D)\neq 0\} = D_I$ and
  \item $E_I \subseteq G_X^1 := \{ D\in \Div(F) : X(D) = 1 \}$.
\end{enumerate}
Finally, a character is \emph{nonsingular}, if it is nontrivial on divisors of degree $0$. For more on this form of characters, we refer the interested reader to \cite{perel2,MR241424} and the references therein.

	\section{A simpler proof of Theorem~\ref{mainthm}}\label{main_proof}
	
   Inspired by the work of \cite{perel2}, we turn our attention to the global rational function field $\F(x)/\F$. Following the discussion of Subsection~\ref{subsec:div}, the places of this function field are exactly the places that correspond to monic irreducible polynomials of $\F$ and $\p_\infty$. In particular, the places of degree $1$ are exactly the ones that correspond to monic linear polynomials and $\p_\infty$. Moreover, for a nonsingular character $X$ modulo $I$, the results of \cite{MR241424}, for this special case of function fields yield the bound
   \begin{equation}\label{eq:perel}
   \left| \sum_{\deg\p = 1} X(\p) \right| \leq (\deg(I)-2) q^{1/2} .
   \end{equation} 

		Now, let $g(x)$ be the minimal polynomial of $\theta$ over $\F$. We define a character $X$ modulo $I=\p_g \, \p_{\infty}^2$ as follows.
		Take some integral $\aaa\in D_{I}$ and take the monic polynomial $f_\aaa\in\F[x]$, such that $(f_\aaa) = \aaa$. Clearly, $g\nmid f_\aaa$, hence $f_\aaa(\theta) \in \Fm^*$. Therefore $\chi(f_\aaa(\theta)) \neq 0$. Next, let $S_\aaa$ be the sum of all the roots of $f_\aaa$, with multiplicities. It is clear that, $S_\aaa \in \F$, therefore assign 
        \[ X (\aaa)  :=   \chi(f_\aaa(\theta)) \psi(-S_\aaa). \]
         Since, for $\aaa, \bb\in D_I$ integral divisors, we have that $f_{\aaa \bb}=f_{\aaa} f_\bb$, we also have that $S_{\aaa\bb} = S_\aaa + S_\bb$, thus $X (\aaa \bb) =X (\aaa) \, X(\bb)$, so that the definition of $X$ can be extended to the whole of $D_I$.
		
		Suppose that $D \in E_I$, where $D=\aaa / \bb$ for $\aaa, \bb \in E_{I}\cap \Prin(\F(x))  $. Since $D \in E_I$, we have that $D$ is linearly equivalent to $1_{\Div(\F(x))}$, which implies, see Corollary~1.4.12 of \cite{stichtenoth09}, that $\deg(D) =0$. Consequently, $\deg(\aaa)=\deg(\bb)$ so that $D=\aaa / \bb=(f_{\aaa}) /(f_{\bb})=(f_{\aaa} / f_{\bb}) = [f_{\aaa} / f_\bb]$. Additionally, $D \in {E}_{I}$ implies $D=[r(x)]$ for some $r(x)$ such that $r(x) \equiv 1 \pmod I$.
        
		Since both $r(x)$ and $f_\aaa/f_\bb$ generate the same ideal, they differ by a nonzero element of $\F$. Hence $f_{\aaa} / f_{\bb} \equiv e \pmod I$ for some $e\in\F^*$. Now, from the definition of $I$, we obtain $\nu_{\infty} (f_\aaa / f_\bb -e) \geq 2$ and $\nu_g (f_{\aaa} / f_{\bb}-e) \geq 1$. {Since $\deg(\aaa) = \deg(\bb)$, we may assume that 
        \[
        f_{\aaa} = x^k + a_1 x^{k-1} + \ldots + a_{k} \ \  \text{and} \ \ f_{\bb} = x^k + b_1 x^{k-1} + \ldots + b_{k}, 
        \]
        where ${a_i}'s$ and ${b_i}'s$ are in $\F$. Expanding $f_\aaa / f_\bb$ formally as a power series in $1/x$ gives
        \[
        \frac{f_\aaa}{f_\bb} - e= c_0 +c_1 \frac{1}{x} + c_2 \frac{1}{x^2} + \ldots,
        \]
        where $c_0 = 1-e$ and $c_1 = a_1 - b_1$. Since the coefficients $a_1$ and $b_1$ are equal to the sums of the roots of $f_{\aaa}$ and $f_{\bb}$, respectively, it follows from $\nu_{\infty}\!\left(f_\aaa / f_\bb - e\right) \ge 2$ that $e = 1$ and $S_{\aaa} = S_{\bb}$.
        Further, $\nu_g (f_{\aaa} / f_{\bb}-e) \geq 1$ implies $f_{\aaa} / f_{\bb}-e = g \cdot h_1 /h_2$, where $h_1, h_2 \in \F[x]$ and $g \nmid h_2$. Thus, given $e=1$, $f_{\aaa}(\theta) = f_{\bb}(\theta)$, hence $X(D) = 1$. Consequently, $X$ is a finite-valued homomorphism from $D_{I}$ to the points of the complex unit circle with kernel containing $E_{I}$. We set $X(\aaa)=0$ for all other divisors $\aaa$.
		
		In order to clarify that $X$ is a nonsingular character, i.e., it is nontrivial on the subgroup of divisors of degree zero, we consider the following cases.  

		  \textbf{Case I}: $\chi$ is nontrivial. 
		  In this case, there exists some $\beta \in \Fm^*$ such that $\chi(\beta) \neq 1$
		  and some $f\in\F[x]$ of degree $\leq n$ such that $f(\theta)=\beta$. Set $f'(x)=x^{q^n-1} + f(x) -1$.  Therefore, $r'(x) := f'(x)/({x^{q^n-1}}+g(x))$ is a rational function of degree $0$ such that $r'(\theta) = \beta$. The divisor corresponding to $r'(x)$ has degree zero and 
        \[
        X((r'(x))) = \chi(r'(\theta)) \cdot \psi(S_{(r'(x))}) , 
        \]
with $\chi(r'(\theta)) = \chi(\beta) \neq 1$. Note that since $\psi$ is an additive character on $\F$, the product $\chi(r'(\theta)) \cdot \psi(S_{(r'(x))})$ can not equal to $1$ as the order of $\chi(r'(\theta))$ is relatively prime to the order of $\psi(S_{(r'(x))})$ in $\mathbb{C}^*$. This implies $X((r'(x))) \neq 1$, that is, $X$ is nonsingular.

\textbf{Case II}: $\psi$ is nontrivial. The map $\F \times \F \rightarrow \F, (t,s)\mapsto t-s$ is surjective. This allow us to choose $t_0, s_0 \in \F$ such that $\psi(t_0-s_0) \neq 1$ as $\psi$ is nontrivial. The rational function $r'(x)=\frac{x-t_0}{x-s_0}$ corresponds to a divisor of degree zero and $\psi(S_{(r'(x))}) \neq 1$. The order argument at the end of Case~I ensures that $X((r'(x))) \neq 1$ and hence $X$ is nonsingular.

	 We have established that $X$ is nonsingular. Now, since the places of degree one are precisely $\p_{\infty}$ and $\p_{\pi_t}$, where $\pi_t(x)=x+t$, and given that $X\left(\p_{\infty}\right)=0$ and $\deg(I) = m+2$, from \eqref{eq:perel},
 		\[
 		\left|\sum_{t \in \F} \chi(\theta+t) \psi(t) \right| \leq m q^{1 / 2} \Rightarrow  \left|\sum_{t \in \F} \chi(\theta+t) \psi(t) \right| \leq m q^{1 / 2} .
 		\]		
The proof of Theorem~\ref{mainthm} is now complete.
	
	\section{The weak line property for primitive normal elements}\label{line}
	
	Let $\theta$ be a generator of the extension $\Fm/\F$ and some $\alpha\in\F$ be such that $\alpha\theta$ is $g$-free, where $g(x) = (x^m-1)/(x-1)$. In this section, we will focus on lines such that $\alpha\theta$ is $(x^m-1)/(x-1)$-free and prove that, for $q$ large enough, every such line contains a primitive normal element. As a consequence of that, we will obtain the weak line property for primitive normal elements for a specific types of extensions.
	%
	
	 So, now suppose that $\theta$ is a generator of $\Fm/\F$ and $\alpha \in \F$, such that $\alpha\theta$ is $(x^m-1)/(x-1)$-free, and set as $\N_{\alpha, \theta}$ the number of elements within the line that are primitive and normal. For our purposes, it suffices to show that $\N_{\alpha,\theta} \neq 0$.
	
	
	\begin{lemma} \label{lemma:N}
		Let $\N_{\alpha,\theta}$ be as defined above. Then
		\[
			\N_{\alpha,\theta} \geq \epsilon(q^m-1)\varepsilon(x^m-1)\left[\frac{q}{\varepsilon\left( \frac{x^m-1}{x-1} \right)} - W(q^m-1)W(x^m-1)q^{1/2} \right],
		\]
		%
%
	where $W(g)$, for $g\in\F[x]$, stands for the number of squarefree monic divisors of $g$ in $\F[x]$ and $W(a)$, for $a\in\mathbb Z$, stands for the number of positive squarefree divisors of $a$.
	\end{lemma}
	\begin{proof}
		By definition we have that
	\[	\N_{\alpha,\theta} = \sum_{x \in \F}\rho_{q^m-1}(\alpha(\theta+x))\kappa_{x^m-1}(\alpha(\theta+x)), \]
	where $\rho_{q^m-1}$ and $\kappa_{x^m-1}$ are the characteristic functions of primitive elements and normal elements, respectively.

		Using \eqref{e-free ch}) and \eqref{g-free ch} we get,
		\begin{align*}
			\N_{\alpha,\theta} &= \epsilon(q^m-1)\varepsilon(x^m-1) \sum_{x \in \F}\sum_{\substack{d\mid q^m-1 \\ f\mid x^m-1}}\frac{\mu(d)\mu'(f)}{\phi(d)\Phi(f)}\sum_{(d), (f)}\chi_d(\alpha(\theta+x))\psi_f(\alpha(\theta+x))  \\
			&= \epsilon(q^m-1)\varepsilon(x^m-1)\sum_{\substack{d\mid q^m-1 \\ f\mid x^m-1}}\frac{\mu(d)\mu'(f)}{\phi(d)\Phi(f)}\sum_{(d), (f)}\chi_d(\alpha)\psi_f(\alpha\theta)\sum_{x \in \F}\chi_d(\theta+x)\psi_f(\alpha x) , 
		\end{align*}
		hence,
		\begin{equation} \label{eq:S1S2S3}
		 \N_{\alpha,\theta} = \epsilon(q^m-1)\varepsilon(x^m-1) [ S_1 + S_2 + S_3 ] , 
		\end{equation}
		where, $S_1$ stands for the part of the sum that corresponds to $d=1$ and $f\mid (x^m-1)/(x-1)$, $S_2$ for the part that corresponds to $d=1$ and $f\nmid (x^m-1)/(x-1)$ and $S_3$ to the part that corresponds to $d\neq 1$.
		
		Regarding, $S_1$, Lemma~\ref{lemma:constant_char} implies that $\psi_f(\alpha x) = \psi_c(\alpha x) = \psi_{c\alpha}(x)$, for some $c\in\Fm$, such that $\Tr_{q^m/q}(c) = 0$. However, $\Tr_{q^m/q}(c\alpha) = \alpha\Tr_{q^m/q}(c) = 0$. Thus, Lemma~\ref{lemma:constant_char} yields $\psi_f(\alpha x)=1$ for all $x\in\F$. It follows that $\chi_1(\theta+x) = \psi_f(\alpha x)=1$ for all $x\in\F$. Thus,
		\begin{equation}\label{eq:S1}
		 S_1 = \sum_{f\mid (x^m-1)/(x-1)} \frac{\mu'(f)}{\Phi(f)} \sum_{(f)} \psi_f(\alpha\theta) \sum_{x\in\F} 1 = \frac{q}{\varepsilon\left( \frac{x^m-1}{x-1} \right)} \cdot \kappa_{(x^m-1)/(x-1)} (\alpha\theta) = \frac{q}{\varepsilon\left( \frac{x^m-1}{x-1} \right)} ,
		 \end{equation}
		since $\alpha\theta$ is $(x^m-1)/(x-1)$-free.
		
		In a similar fashion as above, Lemma~\ref{lemma:constant_char} implies that in the case of $S_2$, we obtain that $\psi_f(\alpha x)$, when seen as an additive character of $\F$, is nontrivial. Therefore,
		\[ S_2 = \sum_{f\nmid (x^m-1)/(x-1)} \frac{\mu'(f)}{\Phi(f)} \sum_{(f)} \psi_f(\alpha\theta) \sum_{x\in\F} \psi_f(\alpha x) \]
		 and Theorem~\ref{mainthm} gives us that the absolute value of the inner sum is bounded by $mq^{1/2}$. This combined with the fact that there are exactly $\Phi(f)$ characters of $\F$-order $f$ implies,
		\begin{equation} \label{eq:S2}
		 |S_2| \leq \left( W(x^m-1) - W\left(\frac{x^m-1}{x-1}\right) \right) m q^{1/2} .
		 \end{equation}
		
		Finally, we have 
		\begin{equation*}
			S_3=\sum_{\substack{d\mid q^m-1, d \neq1 \\ f\mid x^m-1}}\frac{\mu(d)\mu'(f)}{\phi(d)\Phi(f)}\sum_{(d), (f)}\chi_d(\alpha)\psi_f(\alpha\theta)\sum_{x \in \F}\chi_d(\theta+x)\psi_f(\alpha x)
		\end{equation*}We note that in the case of $S_3$, the multiplicative character $\chi_d(\theta+x)$ is non-trivial and Theorem~\ref{mainthm} gives us that the absolute value of the inner sum is bounded by $mq^{1/2}$. This again combined with the fact that there are exactly $\phi(d)$ characters in $\widehat{\mathbb{F}_{q^m}^*}$ of order $d$ and and $\Phi(f)$ characters in $\widehat{\mathbb{F}_{q^m}}$ of $\F$-order $f$ implies, 
		\begin{equation} \label{eq:S3}
		|S_3| \leq (W(q^m-1)-1)W(x^m-1) m q^{1/2} .
		\end{equation}
		The result follows upon combining \eqref{eq:S1S2S3}, \eqref{eq:S1}, \eqref{eq:S2} and \eqref{eq:S3}.
	\end{proof}
	
From the above, one sees that the numbers $W(q^m-1)$ and $W(x^m-1)$ have to be estimated. Regarding $x^m-1$ one easily obtains the trivial bound
\begin{equation} \label{eq:W1}
  W(x^m-1) \leq 2^m .
\end{equation}
Regarding, $W(q^m-1)$, we have that, from Theorem~13.12 of \cite{apostol76},
\[
  W(q^m-1) = o((q^m-1)^{\delta})
\]
for any $\delta > 0$, where $o$ signifies for the little-o notation. In particular, for a fixed integer $m$,
\begin{equation} \label{eq:W2}
W(q^m-1) = o(q^{1/4}).
\end{equation}

We combine \eqref{eq:W1}, \eqref{eq:W2} and the fact that $\varepsilon(g) < 1$ for every $g\in\F[x]$ with Lemma~\ref{lemma:N} and readily obtain that, for $q$ large enough, $\N_{\alpha,\theta}\neq 0$. In other words, we have the following.
\begin{theorem}\label{lineprop_gen}
Fix an integer $m$. There exists some number $\Ll_m$, such that for every prime power $q\geq \Ll_m$, every line of the extension $\Fm/\F$ in which $\alpha\theta$ is $(x^m-1)/(x-1)$-free contains a primitive normal element.
\end{theorem}
%
%
\section{Proof of Theorem~\ref{lineprop}} \label{sec:proof_lineprop}
In this section we briefly explain how Theorem~\ref{lineprop_gen} implies Theorem~\ref{lineprop}.

Let $\Fm/\F$ be an extension, such that $q$ is primitive modulo $m$ and $q\geq \Ll_m$. Let $\theta$ be a generator of that extension and take some $a\in\F^*$. Lemma~\ref{lemma:g-free} yields that $\alpha\theta$ is $(x^m-1)/(x-1)$-free. Theorem~\ref{lineprop_gen} gives us that the aforementioned line contains a primitive normal element. This settles the assertion in Theorem~\ref{lineprop} regarding the weak line property. The assertion about the translate property is immediate.

	\section{Further research}\label{further_research}
	
	We conclude this paper with a short discussion about future developments in the area that was covered in this work.
	
	A first interesting direction would be to try to prove Theorem~\ref{thm:fuwan} (partially or in its full generality) using function field arguments.
	
	Another possible direction is studying the situation when $\alpha\theta$ is not $(x^m-1)/(x-1)$-free, as they were defined in Section~\ref{line}, and/or trying to extend Theorem~\ref{lineprop} to the (nonweak) line property. A careful look in the proof of Lemma~\ref{lemma:N} should indicate the technical reasons behind the restrictions that were put in this work.
	
	A yet different direction would be to generalize the results of this work to the existence of either $r$-primitive and/or $k$-normal elements, where some $\Fm$ is called \emph{$k$-normal} over $\F$ if its $\F$-conjugates produce an $\F$-vector space of dimension $m-k$. Recent effective characterizations for $r$-primitive and $k$-normal elements, that could potentially be applied in this setting, have been established recently in \cite{cohenkapetanakisreis22} and \cite{kapetanakisreis25}, respectively.
	
	From a slightly different point of view, in this work we studied the existence of primitive normal elements of $\Fm/\F$ within the set $\mathcal A_f := \{f(x) : x\in\F\}$, where $f\in\F[x]$ is linear. In \cite{perel2}, the authors work with the prime extension $\mathbb F_{p^m}/\mathbb F_p$ and study the existence of primitive elements within $\mathcal A_f$ where $f\in\mathbb F_p[x]$ is required to be irreducible. In that spirit, we believe that establishing the existence of primitive normal elements of the extension $\Fm/\F$ within $\mathcal A_f$, for irreducible $f\in\F[x]$, would be feasible but nontrivial.
	
	Finally, an intriguing question on this line of research is computing the exact value of the numbers that are proven asymptotically to exist in Theorem~\ref{lineprop}. In this spirit, in 1983 Cohen~\cite{cohen83} showed that $TP(2) = LP(2) = 1$. In 2009, in \cite{cohen09}, the same author showed that $TP(3) = 37$. In 2019, Bailey et al.~\cite{baileycohensutherlandtrudgian19} proved that $LP(3) = 37$ and that $73 \leq TP(4) \leq LP(4) \leq 10282$. Also, in 2020, Cohen and Kapetanakis~\cite{cohenkapetanakis20} established $TP_2(2) = LP_2(2) = 41$. Within the context of Theorem~\ref{lineprop} and given that in the quadratic extension $\mathbb F_{q^2}/\F$ all primitive elements are also normal over $\F$, the fact that $LP(2) = 1$ implies $LPN(2) = 1$. However, computing some of the numbers $TPN(m)$ or $WLPN(m)$, for $m\geq 3$ as in Theorem~\ref{lineprop}, is a nontrivial task that should push theoretical and computational methods to their limits.
		
	\section{Acknowledgments}
	
	The authors are grateful to Igor Shparlinski for a helpful conversation and his useful comments, in particular for pointing out \cite{perel2}. We also thank the reviewers for their thoughtful comments and efforts towards improving the manuscript.


\end{document}